\documentclass{mrlart6-ppn}

\usepackage{amsmath,amsfonts,amsthm,amssymb,amsxtra,setspace,color,graphicx,hyperref}

\newtheorem{theorem}{Theorem}[section]
\newtheorem{lemma}[theorem]{Lemma}
\newtheorem{proposition}[theorem]{Proposition}
\newtheorem{corollary}[theorem]{Corollary}

\newcommand{\R}{{\mathbb R}}
\newcommand{\C}[1]{\mathsf C_{#1}}
\def\1{\mathbb I}
\renewcommand{\(}{\left(}
\renewcommand{\)}{\right)}
\newcommand{\be}[1]{\begin{equation}\label{#1}}
\newcommand{\ee}{\end{equation}}
\newcommand{\ir}[1]{\int_{\R^2}{#1}\;dx}
\newcommand{\ird}[1]{\int_{\R^d}{#1}\;dx}
\newcommand{\irdmu}[1]{\int_{\R^2}{#1}\;d\mu}
\newcommand{\nrm}[2]{\|{#1}\|_{\L^{#2}(\R^d)}}
\newcommand{\nr}[2]{\|{#1}\|_{\L^{#2}(\R^2)}}
\renewcommand{\L}{\mathrm L}

\title[Sobolev and Hardy-Littlewood-Sobolev inequalities]{Sobolev and Hardy-Littlewood-Sobolev inequalities: duality and fast diffusion}
\author[J. Dolbeault]{Jean Dolbeault}
\address{J. Dolbeault: Ceremade, Universit\'e Paris-Dauphine, Place de Lattre de Tassigny, 75775 Paris C\'edex~16, France.}
\email{dolbeaul@ceremade.dauphine.fr}
\date{\today}
\begin{document}
\begin{abstract} In the euclidean space, Sobolev and Hardy-Littlewood-Sobolev inequalities can be related by duality. In this paper, we investigate how to relate these inequalities using the flow of a fast diffusion equation in dimension $d\ge3$. The main consequence is an improvement of Sobolev's inequality when $d\ge5$, which involves the various terms of the dual Hardy-Littlewood-Sobolev inequality. In dimension $d=2$, Onofri's inequality plays the role of Sobolev's inequality and can also be related to its dual inequality, the logarithmic Hardy-Littlewood-Sobolev inequality, by a super-fast diffusion equation. 
\end{abstract}


\maketitle
\thispagestyle{empty}

\section{Introduction}\label{Sec:Intro}

In dimension $d\ge3$, it is well known since E.~Lieb's paper \cite{MR717827} that Hardy-Littlewood-Sobolev inequalities are dual of Sobolev's inequalities. We will investigate this duality using the flow of a fast diffusion equation which has been considered in \cite{delPino-Saez01}. In dimension \hbox{$d=2$}, Onofri's inequality plays the role of Sobolev's inequality and is the dual of the logarithmic Hardy-Littlewood-Sobolev inequality according to~\cite{MR1143664,MR1230930,MR2433703}. Based on \cite{MR1371208,MR1357953}, we will investigate this duality using a variant of the logarithmic diffusion equation, also known as the super-fast diffusion equation.

In a recent paper, \cite{CCL}, E.~Carlen, J.A.~Carrillo and M.~Loss noticed that Hardy-Littlewood-Sobolev inequalities in dimension $d\ge3$ (or the logarithmic Hardy-Little\-wood-Sobolev inequality if $d=2$) and some special Gagliardo-Nirenberg inequalities can be related through another fast diffusion equation. Understanding the differences between the two approaches is one of the motivations of this paper.

\medskip Consider Sobolev's inequality in $\R^d$, $d\ge3$,
\be{Ineq:Sobolev}
\nrm u{2^*}^2\le\mathsf S_d\,\nrm{\nabla u}2^2\quad\forall\;u\in\mathcal D^{1,2}(\R^d)\;,
\ee
where $\mathsf S_d$ is the Aubin-Talenti constant (see \cite{MR0448404,MR0463908}) and $2^*=\frac{2\,d}{d-2}$. The space $\mathcal D^{1,2}(\R^d)$ is defined as the completion of smooth solutions with compact support w.r.t.~the norm $w\mapsto\|w\|:=(\nrm{\nabla w}2^2+\nrm w{2^*}^2)^{1/2}$. The Hardy-Littlewood-Sobolev inequality
\be{Ineq:HLS}
\mathsf S_d\,\nrm v{\frac{2\,d}{d+2}}^2\ge \ird{v\,(-\Delta)^{-1}v}\quad\forall\;v\in\L^\frac{2\,d}{d+2}(\R^d)
\ee
involves the same optimal constant, $\mathsf S_d$.

As it has been noticed in \cite{MR717827},~\eqref{Ineq:Sobolev} and~\eqref{Ineq:HLS} are dual of each other, in the following sense. To a convex functional $F$, we may associate the functional $F^*$ defined by Legendre's duality as
\be{Eqn:Legendre}
F^*[v]:=\sup\(\ird{u\,v}-F[u]\)\;.
\ee
For instance, to $F_1[u]=\frac 12\,\nrm up^2$ defined on $\L^p(\R^d)$, we henceforth associate $F_1^*[v]=\frac 12\,\nrm vq^2$ on $\L^q(\R^d)$ where $p$ and $q$ are H\"older conjugate exponents: $1/p +1/q=1$. The supremum can be taken for instance on all functions in $\L^p(\R^d)$, or, by density, on the smaller space of the functions $u\in\L^p(\R^d)$ such that $\nabla u\in\L^2(\R^d)$. Similarly, to $F_2[u]=\frac 12\,\mathsf S_d\,\nrm{\nabla u}2^2$, we associate $F_2^*[v]=\frac 12\,\mathsf S_d^{-1}\ird{v\,(-\Delta)^{-1}v}$ where $(-\Delta)^{-1}v=G_d*v$ with $G_d(x)=\frac 1{d-2}\,|\mathbb S^{d-1}|^{-1}\,|x|^{2-d}$. As a straightforward consequence of Legendre's duality, if we have a functional inequality of the form $F_1[u]\le F_2[u]$, then we have the dual inequality $F_1^*[v]\ge F_2^*[v]$. 

\medskip In this paper, we investigate the duality of~\eqref{Ineq:Sobolev} and~\eqref{Ineq:HLS} using a nonlinear diffusion equation. If $v$ is a positive solution of the following \emph{fast diffusion} equation:
\be{Eqn:FD}
\frac{\partial v}{\partial t}=\Delta v^m\quad t>0\;,\quad x\in\R^d\,,
\ee
and if we define $\mathsf H(t):=\mathsf H_d[v(t,\cdot)]$, with
\[
\mathsf H_d[v]:=\ird{v\,(-\Delta)^{-1}v}-\mathsf S_d\,\nrm v{\frac{2\,d}{d+2}}^2\,,
\]
then we observe that
\[
\frac 12\,\mathsf H'=-\ird{v^{m+1}}+\mathsf S_d\(\ird{v^\frac{2\,d}{d+2}}\)^\frac 2d\ird{\nabla v^m\cdot\nabla v^\frac{d-2}{d+2}}
\]
where $v=v(t,\cdot)$ is a solution of~\eqref{Eqn:FD}. With the choice $m=\frac{d-2}{d+2}$, we find that $m+1=\frac{2\,d}{d+2}$, so that the above identity can be rewritten with $u=v^m$ as follows.
\begin{proposition}\label{Prop:Hd'} Assume that $d\ge3$ and $m=\frac{d-2}{d+2}$. If $v$ is a solution of~\eqref{Eqn:FD} with nonnegative initial datum in $\L^{2d/(d+2)}(\R^d)$, then
\begin{multline*}
\frac 12\,\frac d{dt}\left[\ird{v\,(-\Delta)^{-1}v}-\mathsf S_d\,\nrm v{\frac{2\,d}{d+2}}^2\right]\\
=\(\ird{v^{m+1}}\)^\frac 2d\left[\mathsf S_d\,\nrm{\nabla u}2^2-\nrm u{2^*}^2\right]\ge0\;.
\end{multline*}\end{proposition}
As a consequence, one can prove that~\eqref{Ineq:HLS}, which amounts to $\mathsf H\le 0$, is a consequence of~\eqref{Ineq:Sobolev}, that is $\mathsf H'\ge 0$, by showing that $\limsup_{t>0}\mathsf H(t)=0$. In this way, we also recover the property that $u=v^m$ is an optimal function for~\eqref{Ineq:Sobolev} if $v$ is optimal for~\eqref{Ineq:HLS}. By integrating along the flow defined by~\eqref{Eqn:FD}, we can actually obtain optimal integral remainder terms which improve on the usual Sobolev inequality~\eqref{Ineq:Sobolev}, but only when $d\ge 5$ for integrability reasons: see Theorem~\ref{Thm:Gap}. A slightly weaker result is the following improved Sobolev inequality, which relates~\eqref{Ineq:Sobolev}~and~\eqref{Ineq:HLS}.
\begin{theorem}\label{Theorem:ExplicitGap} Assume that $d\ge 5$ and let $q=\frac{d+2}{d-2}$. There exists a positive constant~$\mathcal C$ such that, for any $w\in\mathcal D^{1,2}(\R^d)$, we have
\begin{multline*}
\mathsf S_d\,\nrm{w^q}{\frac{2\,d}{d+2}}^2-\ird{w^q\,(-\Delta)^{-1}w^q}\\
\le\mathcal C\,\nrm w{2^*}^\frac 8{d-2}\left[\nrm{\nabla w}2^2-\mathsf S_d\,\nrm w{2^*}^2\right]\;.
\end{multline*}
Moreover, we know that $\mathcal C\le\(1+\frac 2d\)\(1-e^{-d/2}\) \mathsf S_d$.
\end{theorem}
Considerable efforts have been devoted to obtain improvements of Sobolev's inequality starting with \cite{MR790771,MR1124290}. On the whole euclidean space, nice results based on rearrangements have been obtained in \cite{MR2538501} and we refer to \cite{MR2508840} for an interesting review of the various improvements that have been established over the years. It has to be noted that they are all of different nature than the inequality in Theorem~\ref{Theorem:ExplicitGap}, which is the main result of this paper. 

Our approach is based on the evolution equation~\eqref{Eqn:FD}. Many papers have been devoted to the study of the asymptotic behaviour of the solutions, in bounded domains: \cite{MR588035,MR1877973,MR1285092}, or in the whole space: \cite{king1993self,MR1348964,MR1475779}. In particular, the Cauchy-Schwarz inequality has been repeatedly used, for instance in \cite{MR588035,MR1285092}, and turns out to be a key tool in the proof of Theorem~\ref{Theorem:ExplicitGap}, as well as the solution \emph{with separation of variables,} which is related to the Aubin-Talenti optimal function for~\eqref{Ineq:Sobolev}. The novelty in our approach is to consider the problem from the point of view of the functional associated to~\eqref{Ineq:HLS} using\eqref{Eqn:FD} with $m=(d-2)/(d+2)$.

\medskip We now turn our attention to the case of the dimension $d=2$. As we shall see in Section~\ref{Sec:Two}, Onofri's inequality~\cite{MR677001}
\be{Ineq:Onofri}
\log\(\irdmu{e^{\,g}}\)-\irdmu g\le \frac 1{16\,\pi}\,\ir{|\nabla g|^2}\quad\forall\;g\in\mathcal D(\R^d)
\ee
plays the role of Sobolev's inequality in higher dimensions. Here the probability measure $d\mu$ is defined by
\[
d\mu(x):=\mu(x)\,dx\quad\mbox{with}\quad \mu(x):=\frac 1{\pi\,(1+|x|^2)^2}\quad\forall\;x\in\R^2.
\]
As for Sobolev's inequality, duality can also be used. The dual of Onofri's inequality is the logarithmic Hardy-Littlewood-Sobolev inequality: for any $f\in\L^1_+(\R^2)$ with $M=\ir f$, such that $f\,\log f$, $(1+\log|x|^2)\,f\in\L^1(\R^2)$, we have
\be{Ineq:logHLS}
\ir{f\,\log\big(\frac fM\big)}+\frac 2M\int_{\R^2\times\R^2}f(x)\,f(y)\,\log|x-y|\;dx\,dy+M\,\(1+\log\pi\)\ge 0\;.
\ee
The duality has been established on the two-dimensional sphere in \cite{MR1230930,MR1143664} and directly on the euclidean space $\R^2$ in \cite{MR2433703}. The euclidean case can also be recovered from the inequality on the sphere using the stereographic projection. For completeness, we shall give a proof of the inequality (with optimal constants) in the $\R^2$ case: see Proposition~\ref{Prop:LogHLS}. The logarithmic Hardy-Littlewood-Sobolev inequality has recently attracted lots of attention in connection with the Keller-Segel model or in geometry: see for instance \cite{MR2433703,MR2103197,MR2377499}.

We may now proceed in the case $d=2$ as we did for $d\ge3$. Let
\[
\mathsf H_2[v]:=\ir{(v-\mu)\,(-\Delta)^{-1}(v-\mu)}-\frac 1{4\,\pi}\ir{v\,\log\(\frac v\mu\)}\,.
\]
Assume that $v$ is a positive solution of 
\be{Eqn:FDlog}
\frac{\partial v}{\partial t}=\Delta \log\(\frac v\mu\)\quad t>0\;,\quad x\in\R^2\,,
\ee
which replaces~\eqref{Eqn:FD}. Then we have the analog of Proposition~\ref{Prop:Hd'}.
\begin{proposition}\label{Prop:H2'} Assume that $d=2$. If $v$ is a solution of~\eqref{Eqn:FDlog} with nonnegative initial datum $v_0$ in $\L^1(\R^2)$ such that $\ir{v_0}=1$, $v_0\,\log v_0\in\L^1(\R^2)$ and $v_0\,\log \mu\in\L^1(\R^2)$, then
\[
\frac d{dt}\mathsf H_2[v(t,\cdot)]=\frac 1{16\,\pi}\ir{|\nabla u|^2}-\irdmu{\(e^\frac u2-1\) u}
\]
with $\log(v/\mu)=u/2$. \end{proposition}
The right hand side is nonnegative by Onofri's inequality:
\[
\frac d{dt}\mathsf H_2[v(t,\cdot)]\ge\frac 1{16\,\pi}\ir{|\nabla u|^2}+\irdmu u-\log\(\irdmu{e^u}\)\ge 0\;.
\]
See Lemma~\ref{Lem:LogHLSDer} and Corollary~\ref{Cor:LogHLSDer} for details, and Section~\ref{Sec:ImprovedOnofri} for further considerations on the two-dimensional case.

\section{Improved Sobolev inequalities}\label{Sec:Fast2}

This section is devoted to the proof of Theorem~\ref{Theorem:ExplicitGap}. We shall assume that
\[
m=\frac{d-2}{d+2}\quad\mbox{and}\quad d\ge3\;.
\]
{} From the computations of Section~\ref{Sec:Intro}, it is clear that the maximum of $\mathsf H_d[v]$ is achieved if and only if $u=v^m$ is an extremal for Sobolev's inequality. This is of course consistent with the fact that extremal points for Hardy-Littlewood-Sobolev inequalities and Sobolev inequalities are related through Legendre's duality precisely by the relation $u=v^m$.

It is also straightforward to check that~\eqref{Eqn:FD} admits special solutions with \emph{separation of variables} such that, for any $T>0$,
\[
\overline{v}_T(t,x)=c\,(T-t)^\alpha\,(F(x))^\frac{d+2}{d-2}\quad\forall\;(t,x)\in(0,T)\times\R^d\,,
\]
where $\alpha=(d+2)/4$, $c^{1-m}=4\,m\,d$, $m=\frac{d-2}{d+2}$, $p=d/(d-2)$ and $F$ is the Aubin-Talenti solution of $-\Delta F=d\,(d-2)\,F^{(d+2)/(d-2)}$. Such a solution vanishes at $t=T$ and this behaviour is generic in a large class of solutions. Define
\[
\|v\|_*:=\sup_{x\in\R^d}(1+|x|^2)^{d+2}\,|v(x)|\;.
\]
\begin{lemma}\label{Thm:delPinoSaez}{\rm \cite{delPino-Saez01,MR2282669}} For any solution $v$ of~\eqref{Eqn:FD} with nonnegative, continuous, not identically zero initial datum $v_0\in\L^{2d/(d+2)}(\R^d)$, there exists $T>0$, $\lambda>0$ and $x_0\in\R^d$ such that $v(t,\cdot)\not\equiv0$ for any $t\in(0,T)$ and
\[
\lim_{t\to T_-}(T-t)^{-\frac 1{1-m}}\,\|v(t,\cdot)/\overline{v}(t,\cdot)-1\|_*=0
\]
with $\overline{v}(t,x)=\lambda^{(d+2)/2}\,\overline{v}_T(t,(x-x_0)/\lambda)$.\end{lemma}
We shall refer to such a solution as a \emph{solution vanishing at time $T$}. The above result has been established first in \cite{delPino-Saez01} when $\|v_0\|_*$ is finite and extended to solutions corresponding to any initial data $v_0\in\L^{2d/(d+2)}_+(\R^d)$ in \cite[Theorem 7.10]{MR2282669}. In this framework, it is easy to establish further \emph{a priori} estimates as follows.
\begin{lemma}\label{Lem:FDdecay} Let $d\ge3$ and $m=(d-2)/(d+2)$. If $v$ is a solution of~\eqref{Eqn:FD} vanishing at time $T>0$ with initial datum $v_0\in\L^{2d/(d+2)}_+(\R^d)$, then for any $t\in(0,T)$ we have the estimates
\begin{eqnarray*}
&&\(\tfrac{4\,(T-t)}{(d+2)\,\mathsf S_d}\)^\frac d2\le\ird{v^{m+1}(t,x)}\le\ird{v_0^{m+1}}\;,\\
&&\nrm{\nabla v^m(t,\cdot)}2^2\ge\mathsf S_d^{-1}\(\tfrac{4\,(T-t)}{d+2}\)^{\frac d2-1}\;,
\end{eqnarray*}
and the vanishing time $T$ is bounded by
\[
T\le\frac14\,(d+2)\, \mathsf S_d\(\ird{v_0^{m+1}}\)^\frac 2d\,.
\]
If additionnally $d\ge 5$, then $T\ge \frac{d+2}{2\,d}\,\ird{v_0^{m+1}}\,\nrm{\nabla v_0^m}2^{-2}$ and
\begin{eqnarray*}
&&\ird{v^{m+1}(t,x)}\ge\ird{v_0^{m+1}}-\tfrac {2\,d}{d+2}\,t\,\nrm{\nabla v_0^m}2^2\;,\\
&&\nrm{\nabla v^m(t,\cdot)}2^2\le\nrm{\nabla v_0^m}2^2\;.
\end{eqnarray*}
\end{lemma}
\begin{proof} By definition of the vanishing time, we find that $\mathsf J(t):=\ird{v(t,x)^{m+1}}$ satisfies $\mathsf J(t)>0$ for any $t\in(0,T)$, $\mathsf J(T)=0$ and, using the equation and integrating by parts,
\[
\mathsf J'=-(m+1)\,\nrm{\nabla v^m}2^2\le -\frac{m+1}{\mathsf S_d}\,\mathsf J^{1-\frac 2d}\;.
\]
If $d\ge5$, then we also have
\[
\mathsf J''=2\,m\,(m+1)\ird{v^{m-1}\,(\Delta v^m)^2}\ge 0\;.
\]
It is easy to check that such an estimate makes sense if $v=\overline{v}_T$. For a general solution, this is also true as can be seen by rewriting the problem on $\mathbb S^d$ as in \cite{delPino-Saez01}. It is then clear that integrability conditions for $v$ are exactly the same as for $\overline{v}_T$.

By integrating the first inequality from $t$ to $T$, we find that
\[
-\mathsf J(t)^\frac2d=\mathsf J(T)^\frac2d-\mathsf J(t)^\frac2d\le-\tfrac{4\,(T-t)}{(d+2)\,\mathsf S_d}\quad\forall\;t\in[0,T)\;,
\]
which gives an upper bound for $T$ by choosing $t=0$. The second inequality shows the decay of $t\mapsto\nrm{\nabla v^m(t,\cdot)}2^2$ and gives the estimate $0=\mathsf J(T)\ge \mathsf J(0)+T\,\mathsf J'(0)$, thus providing a lower bound for $T$.\end{proof}

For later purpose, let us notice that
\be{Ineq:Kappa}
\frac{\mathsf J'}{\mathsf J}\le -\frac{m+1}{\mathsf S_d}\,\mathsf J^{-\frac 2d}\le -\kappa\quad\mbox{with}\quad\kappa:=\frac{2\,d}{d+2}\,\frac1{\mathsf S_d}\(\ird{v_0^{m+1}}\)^{\!-\frac 2d}\le \frac d{2\,T}\;.
\ee

\medskip Next, we are going to compute the second derivative of $\mathsf H(t)=\mathsf H_d[v(t,\cdot)]$ w.r.t.~$t$ along the flow of~\eqref{Eqn:FD}. For this purpose, we assume that $d\ge 5$ and  notice that by the Cauchy-Schwarz inequality, we have
\begin{multline*}
\nrm{\nabla v^m}2^4=\(\ird{v^{(m-1)/2}\,\Delta v^m\cdot v^{(m+1)/2}}\)^2\\
\le\ird{v^{m-1}\,(\Delta v^m)^2}\ird{v^{m+1}}\;.
\end{multline*}
As a consequence, we get that $\mathsf Q(t):=\nrm{\nabla v^m(t,\cdot)}2^2\(\ird{v^{m+1}(t,x)}\)^{-(d-2)/d}$ is monotone decreasing. More precisely, with $m=(d-2)/(d+2)$,
\be{Eqn:LambdaG}
\Lambda(t):=\frac {\ird{|\nabla(v(t,x))^m|^2}}{\ird{(v(t,x))^{m+1}}}\quad\mbox{and}\quad\mathsf G(t_1,t_2):=\exp\left[(m+1)\int_{t_1}^{t_2} \Lambda(s)\,ds\right]\,,
\ee
we get that $\mathsf Q'=-2\,m\,\mathsf J^{\frac 2d-1}\,\mathsf K$ with $\mathsf K:=\ird{v^{m-1}\,|\Delta v^m+\Lambda\,v|^2}$ and, from
\[
\mathsf H'=2\,\mathsf J\,(\mathsf S_d\,\mathsf Q-1)\;,
\]
we deduce that
\be{Eqn:SecondDer}
\mathsf H''=\frac{\mathsf J'}{\mathsf J}\,\mathsf H'+2\,\mathsf J\, \mathsf S_d\,\mathsf Q'=-(m+1)\,\Lambda\,\mathsf H'-4\,m\,\mathsf S_d\,\mathsf J^\frac 2d\,\mathsf K\;.
\ee
As a consequence, the standard Sobolev inequality can be improved by \emph{an integral remainder~term.}
\begin{theorem}\label{Thm:Gap} Assume that $d\ge 5$ and let $q=\frac{d+2}{d-2}$. For any $w\in\mathcal D^{1,2}(\R^d)$ such that $\|w^q\|_*<\infty$, we have
\begin{multline*}
\mathsf S_d\,\nrm{w^q}{\frac{2\,d}{d+2}}^2-\ird{w^q\,(-\Delta)^{-1}w^q}+\frac 4q\,\,\mathsf S_d\,\int_0^Tdt\int_0^t\mathsf J^\frac 2d(s)\,\mathsf K(s)\,\mathsf G(t,s)\;ds\\
=2\,\nrm w{2^*}^\frac 4{d-2}\left[\nrm{\nabla w}2^2-\mathsf S_d\,\nrm w{2^*}^2\right]\,\int_0^T\kern -4pt\mathsf G(t,0)\;dt\;.
\end{multline*}
Here $v$ is the solution of~\eqref{Eqn:FD} with $v_0=w^q$, $m=(d-2)/(d+2)=1/q$, $T$ is the vanishing time, $\Lambda$ and $G$ are defined by~\eqref{Eqn:LambdaG}, and we recall that $\mathsf J(t)=\ird{v^{m+1}(t,x)}$, $\mathsf K(t)=\ird{v^{m-1}(t,x)\,|\Delta v^m(t,x)+\Lambda(t)\,v(t,x)|^2}$.
\end{theorem}
\begin{proof} The identity follows from~\eqref{Eqn:SecondDer} after an integration from $0$ to $t$ and another one from $0$ to $T$. Details are left to the reader. \end{proof}

Recall that the standard Hardy-Littlewood-Sobolev inequality~\eqref{Ineq:HLS} amounts to
\[
\mathsf S_d\,\nrm{w^q}{\frac{2\,d}{d+2}}^2-\ird{w^q\,(-\Delta)^{-1}w^q}=-\mathsf H_d[w^q]\ge 0\;.
\]
The main drawback of Theorem~\ref{Thm:Gap} is that several quantities can be computed only through the evolution equation and are therefore non explicit. With simple estimates, we can however get rid of such quantities. This is the purpose of Theorem~\ref{Theorem:ExplicitGap}.

\begin{proof} [Proof of Theorem~\ref{Theorem:ExplicitGap}] Notice that $\nrm{w^q}{2d/(d+2)}^2=\nrm w{2^*}^{1+2/d}$, so that the inequality holds in the space $\mathcal D^{1,2}(\R^d)$. Theorem~\ref{Theorem:ExplicitGap} can be established first in the setting of smooth functions such that $\|w^q\|_*<\infty$ and then arguing by density.

{} From~\eqref{Ineq:Kappa} and~\eqref{Eqn:SecondDer}, we know that
\[
\mathsf H''\le-\kappa\,\mathsf H'\quad\mbox{with}\quad\kappa=\frac{2\,d}{d+2}\,\frac1{\mathsf S_d}\(\ird{v_0^{m+1}}\)^{-2/d}\,.
\]
By writing that $-\mathsf H(0)=\mathsf H(T)-\mathsf H(0)\le\mathsf H'(0)\,(1-e^{-\kappa\,T})/\kappa$ and using the estimate $\kappa\,T\le d/2$, we obtain the result.\end{proof}

\section{The two-dimensional case}\label{Sec:Two}

\subsection{From Gagliardo-Nirenberg inequalities to Onofri's inequality}\label{Sec:Onofri}

Consider the following sub-family of Gagliardo-Nirenberg inequalities
\be{Ineq:GN}
\nrm f{2p}\le\mathsf C_{p,d}\,\nrm{\nabla f}2^\theta\,\nrm f{p+1}^{1-\theta}
\ee
with $\theta=\theta(p):=\frac{p-1}p\,\frac d{d+2-p\,(d-2)}$, $1<p\le\frac d{d-2}\;\mbox{if}\;d\ge3$ and $1<p<\infty\;\mbox{if}\;d=2$. Such an inequality holds for any smooth function $f$ with sufficient decay at infinity and, by density, for any function $f\in \L^{p+1}(\R^d)$ such that $\nabla f$ is square integrable. We shall assume that $\mathsf C_{p,d}$ is the best possible constant. In \cite{MR1940370}, it has been established that equality holds in \eqref{Ineq:GN} if $f=F_p$ with
\be{Eqn:Optimal}
F_p(x)=(1+|x|^2)^{-\frac 1{p-1}}\quad\forall\;x\in\R^d
\ee
and that all extremal functions are equal to $F_p$ up to a multiplication by a constant, a translation and a scaling. If $d\ge 3$, the limit case $p=d/(d-2)$ corresponds to Sobolev's inequality and one recovers the results of T.~Aubin and G.~Talenti in~\cite{MR0448404,MR0463908}, with $\theta=1$: the optimal functions for~\eqref{Ineq:Sobolev} are, up to scalings, translations and multiplications by a constant, all equal to $F_{d/(d-2)}(x)=(1+|x|^2)^{-(d-2)/2}$, and
\[
\mathsf S_d=(\C{d/(d-2),\kern 1pt d})^2\,.
\]
When $p\to 1$, the inequality becomes an equality, so that we may differentiate both sides with respect to $p$ and recover the euclidean logarithmic Sobolev inequality in optimal scale invariant form (see \cite{Gross75,MR479373,MR1940370} for details). For completeness, let us mention that the fast diffusion equation~\eqref{Eqn:FD} has deep connection with~\eqref{Ineq:GN} when $p=1/(2\,m-1)$, but this is out of the scope of the present paper: see \cite{MR1940370} for more details on this issue.

We now investigate the limit $p\to\infty$ in~\eqref{Ineq:GN} when $d=2$: Onofri's inequality~\eqref{Ineq:Onofri} can be obtained in the limit. As an endpoint case of the family of the Gagliardo-Nirenberg inequalities~\eqref{Ineq:GN}, it plays the role of Sobolev's inequality in dimension~$d\ge3$.
\begin{proposition}\label{Prop:Onofri} Assume that $g\in\mathcal D(\R^d)$ is such that $\irdmu g=0$ and let $f_p:=F_p(1+\frac g{2p})$, where $F_p$ is defined by~\eqref{Eqn:Optimal}. Then we have
\[
1\le\lim_{p\to\infty}\mathsf C_{p,2}\,\frac{\nr{\nabla f}2^{\theta(p)}\,\nr f{p+1}^{1-\theta(p)}}{\nr f{2p}}=\frac{e^{\frac 1{16\,\pi}\,\ir{|\nabla g|^2}}}{\irdmu{e^{\,g}}}\;.
\]
\end{proposition}
\noindent We recall that $\mu(x):=\frac 1\pi\,(1+|x|^2)^{-2}$, and $d\mu(x):=\mu(x)\,dx$.
\begin{proof} We can rewrite~\eqref{Ineq:GN} as
\[
\frac{\ir{|f|^{2p}}}{\ir{|F_p|^{2p}}}\le\(\frac{\ir{|\nabla f|^2}}{\ir{|\nabla F_p|^2}}\)^\frac{p-1}2\,\frac{\ir{|f|^{p+1}}}{\ir{|F_p|^{p+1}}}
\]
and observe that, with $f=f_p$, we have:\\
(i) $\lim_{p\to\infty}\ir{|F_p|^{2p}}=\ir{\frac 1{(1+|x|^2)^2}}=\pi$ and
\[
\lim_{p\to\infty}\ir{|f_p|^{2p}}=\ir{F_p^{2p}\,(1+\tfrac g{2p})^{2p}}=\ir{\frac{e^g}{(1+|x|^2)^2}}
\]
so that the $\ir{|f|^{2p}}/\ir{|F_p|^{2p}}$ converges to $\irdmu{e^{\,g}}$ as $p\to\infty$,\\
(ii) $\ir{|F_p|^{p+1}}=(p-1)\,\pi/2$, $\lim_{p\to\infty}\ir{|f_p|^{p+1}}=\infty$, but
\[
\lim_{p\to\infty}\frac{\ir{|f_p|^{p+1}}}{\ir{|F_p|^{p+1}}}=1\;,
\]
(iii) expanding the square and integrating by parts, we find that
\begin{multline*}
\ir{|\nabla f_p|^2}=\frac 1{4p^2}\ir{F_p^2\,|\nabla g|^2}-\ir{(1+\tfrac g{2p})^2\,F_p\,\Delta F_p}\\=\frac 1{4p^2}\ir{|\nabla g|^2}+\frac{2\pi}{p+1}+o(p^{-2})\;.
\end{multline*}
Here we have used $\ir{|\nabla F_p|^2}=\frac{2\pi}{p+1}$ and the condition $\irdmu g=0$ in order to discard one additional term of the order of $p^{-2}$. On the other hand, we find that
\[
\(\frac{\ir{|\nabla f|^2}}{\ir{|\nabla F_p|^2}}\)^\frac{p-1}2\sim\(1+\frac{p+1}{8\pi\,p^2}\ir{|\nabla g|^2}\)^\frac{p-1}2\sim e^{\frac 1{16\,\pi}\,\ir{|\nabla g|^2}}
\]
as $p\to\infty$. Collecting these estimates concludes the proof.\end{proof}

\subsection{Legendre duality}\label{Sec:Legendre}

Now we study the duality which relates Onofri's inequality~\eqref{Ineq:Onofri} and the logarithmic Hardy-Littlewood-Sobolev inequality~\eqref{Ineq:logHLS} in $\R^2$. With
\[
F_1[u]:=\log\(\irdmu{e^u}\)\quad\mbox{and}\quad F_2[u]:=\frac 1{16\,\pi}\ir{|\nabla u|^2}+\ir{u\,\mu}\;,
\]
Onofri's inequality amounts to $F_1[u]\le F_2[u]$. As for the case $d\ge 3$, we use Legendre's transformation~\eqref{Eqn:Legendre}.
\begin{proposition}\label{Prop:LogHLS} For any $v\in\L^1_+(\R^2)$ with $\ir v=1$, such that $v\,\log v$ and $(1+\log|x|^2)\,v\in\L^1(\R^2)$, we have
\[
F_1^*[v]-F_2^*[v]=\ir{v\,\log\(\frac v\mu\)}-4\,\pi\ir{(v-\mu)\,(-\Delta)^{-1}(v-\mu)}\ge0\;.
\]
\end{proposition}
This result has been observed by several authors: see for instance \cite{MR1230930,MR1143664} for the duality argument on the two-dimensional sphere, which gives the above inequality on~$\R^2$ by the stereographic projection. For completeness, let us give a proof. We essentially follow the computation of \cite[Appendix]{MR2433703} but explicitly compute the constants. 

\begin{proof} Taking the infimum in~\eqref{Eqn:Legendre} on the Orlicz space
\[
\left\{u\in\L^1_{\rm loc}(\R^2)\,:\,\nabla u\in\L^2(\R^2)\,,\;\irdmu{e^u}<\infty\right\}\;,
\]
we find that
\begin{eqnarray*}
F_1^*[v]&\kern-7pt=&\kern-7pt\ir{u\,v}-\log\(\irdmu{e^u}\)\quad\mbox{with}\quad \log v=\log \mu+u-\log\(\irdmu{e^u}\)\,,\\
F_2^*[v]&\kern-7pt=&\kern-7pt\ir{u\,(v-\mu)}-\frac 1{16\,\pi}\ir{|\nabla u|^2}\quad\mbox{with}\quad -\Delta u=8\,\pi\,(v-\mu)\;,
\end{eqnarray*}
which proves the result. \end{proof}

A useful observation is the fact that $-\Delta\log \mu=8\,\pi\,\mu$ can be inverted as 
\[
(-\Delta)^{-1}\mu=\frac 1{8\,\pi}\,\log \mu+C\;.
\]
It is then easy to check that $u:=(-\Delta)^{-1}\mu=G_2*\mu$, with $G_2(x)=-\frac 1{2\pi}\,\log |x|$, is such that $u(0)=0$, which determines $C=\frac 1{8\,\pi}\,\log\pi$. Hence with the notations of Proposition~\ref{Prop:LogHLS}, we may observe that
\begin{eqnarray*}
F_1^*[v]&\kern-7pt=&\kern-7pt\ir{v\,\log\(\frac v\mu\)}\;,\\
F_2^*[v]&\kern-7pt=&\kern-7pt4\,\pi\ir{v\,(-\Delta)^{-1}v}-\ir{v\,\log \mu}-1-\log\pi\;.
\end{eqnarray*}
Collecting these observations, we find that
\[
F_1^*[v]-F_2^*[v]=\ir{v\,\log v}-4\,\pi\ir{v\,(-\Delta)^{-1}v}+1+\log\pi\ge0\;.
\]
For any $f\in\L^1_+(\R^2)$ with $M=\ir f$, such that $f\,\log f$, $(1+\log|x|^2)\,f\in\L^1(\R^2)$, this inequality, written for $v=f/M$, is nothing else than the logarithmic Hardy-Littlewood-Sobolev inequality~\eqref{Ineq:Onofri}.

\subsection{Logarithmic diffusion equation: Onofri and logarithmic Hardy-Little\-wood-Sobolev inequalities}\label{Sec:FastLog}

The computation of $\frac d{dt}\mathsf H_2[v(t,\cdot)]$ in Section~\ref{Sec:Intro} is formal but can easily be justified after noticing that the image $w$ of $v$ by the inverse stereographic projection on the sphere $\mathbb S^2$, up to a scaling, solves the equation
\[
\frac{\partial w}{\partial t}=\Delta_{\mathbb S^2}w\;.
\]
More precisely, if $x=(x_1,x_2)\in\R^2$, then $u$ and $w$ are related by
\[
w(t,y)=\frac{u(t,x)}{4\,\pi\,\mu(x)}\;,\quad y=\(\tfrac{2\,(x_1,x_2)}{1+|x|^2},\tfrac{1-|x|^2}{1+|x|^2}\)\in\mathbb S^2\,.
\]
See \cite[Section 8.2]{MR2282669} for a review of some known results for the logarithmic diffusion equation, or super-fast diffusion equation, $\frac{\partial v}{\partial t}=\Delta \log v$, and \cite{MR1291536,MR1357953} for earlier results. Necessary adaptations to the case of~\eqref{Eqn:FDlog} are left to the reader. Striking properties are the facts that the solution of~\eqref{Eqn:FDlog} globally exists and its mass is preserved. For simplicity, we shall therefore assume that $1=\ir{v_0}=\ir{v(t,x)}$ for any $t\ge 0$ and recall that $\ir\mu=1$. It turns out that $\frac d{dt}\mathsf H_2[v(t,\cdot)]$ has a sign because of Onofri's inequality. 
\begin{lemma}\label{Lem:LogHLSDer} For any $u\in\mathcal D(\R^d)$ such that $\irdmu{e^\frac u2}=1$, we have
\[
\frac 1{16\,\pi}\ir{|\nabla u|^2}\ge\irdmu{\(e^\frac u2-1\) u}\;.
\]
\end{lemma}
This inequality can of course be extended by density to the natural Orlicz space for which all integrals are well defined.

\begin{proof} Inspired by~\cite{MR1796718}, let us define
\[
h(t):=\log\(\irdmu{e^{\kern 0.5pt t\kern 0.5pt u}}\)\;.
\]
\emph{Claim 1: $h$ is convex and $h'$ is convex.} Let us observe that
\[
\(\irdmu{e^{\kern 0.5pt t\kern 0.5pt u}}\)^2h''(t)=\irdmu{e^{\kern 0.5pt t\kern 0.5pt u}}\irdmu{u^2\,e^{\kern 0.5pt t\kern 0.5pt u}}-\(\irdmu{u\,e^{\kern 0.5pt t\kern 0.5pt u}}\)^2\ge0
\]
by the Cauchy-Schwarz inequality. With one more derivation, we find that
\begin{multline*}
\(\irdmu{e^{\kern 0.5pt t\kern 0.5pt u}}\)^3h'''(t)\\
=\irdmu{u^3\,e^{\kern 0.5pt t\kern 0.5pt u}}\(\irdmu{e^{\kern 0.5pt t\kern 0.5pt u}}\)^2-3\irdmu{u^2\,e^{\kern 0.5pt t\kern 0.5pt u}}\irdmu{u\,e^{\kern 0.5pt t\kern 0.5pt u}}\irdmu{e^{\kern 0.5pt t\kern 0.5pt u}}\\
+2\(\irdmu{u\,e^{\kern 0.5pt t\kern 0.5pt u}}\)^3\,.
\end{multline*}
Let us prove that $h'''(t)>0$ for $t\in(0,1)$ a.e. If $\irdmu{u\,e^{\kern 0.5pt t\kern 0.5pt u}}>0$, define
\[
d\nu_t:=\frac{e^{\kern 0.5pt t\kern 0.5pt u}}{\irdmu{e^{\kern 0.5pt t\kern 0.5pt u}}}\,d\mu\quad\mbox{and}\quad v:=\frac u{\int_{\R^2}u\;d\nu_t}\;,
\]
so that $\int_{\R^2}v\,d\nu_t=1$. Then $h'''(t)$ has the sign of \hbox{$\int_{\R^2}(v^3-3\,v^2+2)\,d\nu_t$}. Taking the constraint into account, a direct optimization shows that an optimal function $v$ takes two values, $1\pm a$ for some $a>0$ and moreover, $\nu_t(\{v=1+a\})=\nu_t(\{v=1-a\})$ because of the condition $\int_{\R^2}v\,d\nu_t=1$. In such a case, it is straightforward to check that \hbox{$\int_{\R^2}(v^3-3\,v^2+2)\,d\nu_t=0$}, thus proving that $h'''(t)\ge0$.

If $\irdmu{u\,e^{\kern 0.5pt t\kern 0.5pt u}}<0$, a similar computation with $v:=-u/\int_{\R^2}u\,d\nu_t$ shows that $-h'''(t)$ has the sign of $\int_{\R^2}(-v^3-3\,v^2+2)\,d\nu_t$ under the condition $\int_{\R^2}v\,d\nu_t=-1$, thus proving again that $h'''(t)\ge0$.

Since $h'(t)=\int_{\R^2}u\,d\nu_t$ is monotone increasing, $\irdmu{u\,e^{\kern 0.5pt t\kern 0.5pt u}}=0$ occurs for at most one $t\in[0,1]$ and we can conclude that $h'$ is convex. 

\medskip\noindent\emph{Claim 2: $h(1)\ge h'(1/2)$.} We know that $h(0)=h(1/2)=0$ and, by convexity of $h'$, $h''$ is monotone nondecreasing. Let $\kappa:=h''(1/2)$ and $\mathsf p:=h'(1/2)$. On the one hand, we have $h''(t)\le\kappa$ for any $t\in[0,1/2]$, which means that, after integrating from $t<1/2$ to $1/2$, we have
\[
\mathsf p-h'(t)\le\kappa\,(\tfrac 12-t)\quad\forall\;t\in[0,\tfrac 12]\;.
\]
Hence, one more integration from $0$ to $1/2$ gives 
\[
\frac{\mathsf p}2=\left[\mathsf p\,t-h(t)\right]_0^{1/2}\le\left[\kappa\, \tfrac t2\,(1-t)\right]_0^{1/2}=\frac\kappa8\;.
\]
On the other hand, we have $h''(t)\ge\kappa$ for any $t\in[1/2,1]$, which means that, after integrating from $1/2$ to $t>1/2$, we have
\[
h'(t)-\mathsf p\ge\kappa\,(t-\tfrac 12)\quad\forall\;t\in[\tfrac 12,1]\;.
\]
One more integration from $1/2$ to $1$ gives 
\[
h(1)-\frac{\mathsf p}2\ge\left[\kappa\, \tfrac t2\,(t-1)\right]_0^{1/2}=\frac\kappa8\;.
\]
Collecting the two estimates, we find that $h(1)\ge\mathsf p$, which proves the claim. 

\medskip Using $h(0)=\log\(\irdmu{}\)=0$ and $h(1/2)=\log\(\irdmu{e^{u/2}}\)=0$, we have found that $\mathsf p=h'(1/2)=\irdmu{u\,e^{u/2}}\le h(1)=\log\(\irdmu{e^u}\)$. This proves that
\[
\frac d{dt}\mathsf H_2[v(t,\cdot)]\ge\frac 1{16\,\pi}\ir{|\nabla u|^2}+\irdmu u-\log\(\irdmu{e^u}\)\ge 0\;,
\]
where the last inequality is nothing else than Onofri's inequality.\end{proof}
\begin{corollary}\label{Cor:LogHLSDer} Consider a solution of~\eqref{Eqn:FDlog} a nonnegative initial datum $v_0$ such that $\ir{v_0}=1$. Then $\frac d{dt}\mathsf H_2[v(t,\cdot)]\ge 0$.\end{corollary}

\subsection{Towards an improved Onofri inequality ?}\label{Sec:ImprovedOnofri}

At least at a formal level, we can observe that computations of Section~\ref{Sec:Fast2} for $d\ge3$ do not carry to the case $d=2$. With $u=2\,\log(v/\mu)$, Equation~\eqref{Eqn:FDlog} becomes
\[
\mu\,\frac{\partial u}{\partial t}=e^{-\frac u2}\,\Delta u\;,
\]
so that $\mathsf J:=2\ir{v\,\(\log\(v/\mu\)-1\)}$ takes the form
\[
\mathsf J=\irdmu{e^\frac u2\,(u-2)}
\]
and, when computing along the flow, we get
\[
\mathsf J'=-\frac 12\ir{|\nabla u|^2}\quad\mbox{and}\quad\mathsf J''=\int_{\R^2}(\Delta u)^2\,e^{-\frac u2}\;\frac{dx}\mu\;.
\]
If we consider the Cauchy-Schwarz inequality as in Section~\ref{Sec:Fast2}, namely
\begin{multline*}
4\,{\mathsf J'}^2=\(\ir{|\nabla u|^2}\)^2=\(\ir{\Delta u\,u}\)^2\\
=\(\ir{\Delta u\,e^{-u/4}\,\frac1{\sqrt \mu}\cdot u\,e^{u/4}\,\sqrt \mu}\)^2\le\mathsf J''\irdmu{u^2\,e^\frac u2}\,,
\end{multline*}
it clearly turns out that the above r.h.s.~cannot be controlled by $\mathsf J\,\mathsf J''$ using a simple Cauchy-Schwarz inequality. The scheme of the proof of Theorem~\ref{Theorem:ExplicitGap} for improving Sobolev's inequality cannot be directly applied to Onofri's inequality.

\section{Concluding remarks}\label{Sec:Conclusion}

In \cite{CCL}, E.~Carlen, J.A.~Carrillo and M.~Loss noticed that Hardy-Littlewood-Sobolev inequalities and Gagliardo-Nirenberg inequalities can be related through the fast diffusion equation~\eqref{Eqn:FD} with exponent $m=d/(d+2)$, when $d\ge 3$. The key computation goes as follows:
\begin{multline*}
\frac 12\,\frac d{dt}\mathsf H_d[v(t,\cdot)]=\frac 12\,\frac d{dt}\left[\ird{v\,(-\Delta)^{-1}v}-\mathsf S_d\,\nrm v{\frac{2\,d}{d+2}}^2\right]\\
=\tfrac{d\,(d-2)}{(d-1)^2}\,\mathsf S_d\,\nrm u{q+1}^{4/(d-1)}\,\nrm{\nabla u}2^2-\nrm u{2q}^{2q}
\end{multline*}
with $u=v^{(d-1)/(d+2)}$ and $q=(d+1)/(d-1)$. An explicit computation shows that
\[
\tfrac{d\,(d-2)}{(d-1)^2}\,\mathsf S_d=(\C{q,\kern 1pt d})^{2q}\,,
\]
thus proving that $\frac d{dt}\mathsf H_d[v(t,\cdot)]$ has a sign because of the Gagliardo-Nirenberg inequalities~\eqref{Ineq:GN}. Using the fact that the asymptotic behaviour of the solutions of the fast diffusion equation~\eqref{Eqn:FD} with $m=d/(d+2)$ is governed by the Barenblatt self-similar solutions (see \cite{MR1940370}), an integral remainder term (obtained by integrating along the flow of~\eqref{Eqn:FD}) has been established, which improves on Hardy-Littlewood-Sobolev inequalities~\eqref{Ineq:HLS}. 

A similar result holds for the logarithmic Hardy-Littlewood-Sobolev inequality in dimension $d=2$. Computing along the flow of~\eqref{Eqn:FD} with exponent $m=1/2$, it turns out that
\begin{multline*}
\frac{\nr v1}8\,\frac d{dt}\left[\frac{4\,\pi}{\nr v1}\ir{v\,(-\Delta)^{-1}v}-\ir{v\,\log v}\right]\\
=\nr u4^4\,\nr{\nabla u}2^2-\pi\,\nr v6^6\;,
\end{multline*}
which is again one of the Gagliardo-Nirenberg inequalities~\eqref{Ineq:GN}, the one corresponding to $d=2$ and $q=3$, which is moreover such that $\pi\,(\C{3,\kern 1pt 2})^6=1$. 

The results of \cite{CCL} relate Hardy-Littlewood-Sobolev inequalities (the logarithmic Hardy-Littlewood-Sobolev inequality if $d=2$) with non-critical Gagliardo-Nirenberg inequalities through an evolution equation, although the Barenblatt self-similar solutions associated to the diffusion equation are not optimal for such Gagliardo-Nirenberg inequalities. In such a setting as well as in the setting considered in Sections~\ref{Sec:Intro}--\ref {Sec:Two}, a nonlinear flow allows to improve on well known inequalities, with rather straightforward computations. These two examples suggest that much more can be done using flows of nonlinear diffusion equations.

\par\medskip\centerline{\rule{2cm}{0.2mm}}\medskip\begin{spacing}{0.5} \noindent{\small{\bf Acknowlegments.} The author thanks J.A.~Carrillo and M.J.~Esteban for useful suggestions. This work has been partially supported by the projects CBDif and EVOL of the French National Research Agency (ANR).
\par\medskip\noindent\copyright\,2010 by the author. This paper may be reproduced, in its entirety, for non-commercial purposes.}\end{spacing}


\end{document}